\documentclass[twoside,11pt]{article}

\usepackage{amsfonts, amsmath, amssymb, amsthm, constants, bbm}
\usepackage{graphicx, subfigure}
\usepackage[top=1in,bottom=1in,left=1in,right=1in]{geometry}
\usepackage[sort&compress]{natbib} 
\usepackage{pstricks,pst-plot,pst-node}

\usepackage{xcolor}
\definecolor{darkred}{RGB}{150,0,0}
\definecolor{darkgreen}{RGB}{0,100,0}
\definecolor{darkblue}{RGB}{0,0,200}

\usepackage{hyperref}
\hypersetup{colorlinks=true, linkcolor=darkred, citecolor=darkgreen, urlcolor=darkblue}
\usepackage{url}

\newtheorem{thm}{Theorem}
\newtheorem{prp}{Proposition}
\newtheorem{lem}{Lemma}

\theoremstyle{remark}
\newtheorem{rem}{Remark}

\def\beq{\begin{equation}}
\def\eeq{\end{equation}}
\def\beqn{\begin{eqnarray*}}
\def\eeqn{\end{eqnarray*}}
\def\bitem{\begin{itemize}}
\def\eitem{\end{itemize}}
\def\benum{\begin{enumerate}}
\def\eenum{\end{enumerate}}

\newcommand{\thmref}[1]{Theorem~\ref{thm:#1}}
\newcommand{\prpref}[1]{Proposition~\ref{prp:#1}}

\newcommand{\lemref}[1]{Lemma~\ref{lem:#1}}
\newcommand{\secref}[1]{Section~\ref{sec:#1}}

\newcommand{\figref}[1]{Figure~\ref{fig:#1}}
\newcommand{\remref}[1]{Remark~\ref{rem:#1}}

\DeclareMathOperator*{\argmin}{arg\, min}
\DeclareMathOperator{\diam}{diam}
\DeclareMathOperator{\dist}{dist}

\DeclareMathOperator{\acos}{acos}
\DeclareMathOperator{\asin}{asin}

\def\cB{\mathcal{B}}

\def\cE{\mathcal{E}}

\def\cH{\mathcal{H}}

\def\cX{\mathcal{X}}

\def\bbR{\mathbb{R}}

\newcommand{\E}{\operatorname{\mathbb{E}}}
\renewcommand{\P}{\operatorname{\mathbb{P}}}

\newcommand{\pr}[1]{\mathbb{P}\left(#1\right)}

\newcommand{\<}{\langle}
\renewcommand{\>}{\rangle}

\def\eps{\varepsilon}

\def\comp{\mathsf{c}}
\newcommand{\overarc}[1]{{\rm arc}{(#1)}}
\def\1{\mathbbm{1}}

\defcitealias{R}{R}

\begin{document}		

\title{On Estimating the Perimeter Using the Alpha-Shape}
\author{
Ery Arias-Castro\footnote{Department of Mathematics, University of California, San Diego \{\href{http://math.ucsd.edu/~eariasca}{http://math.ucsd.edu/{\tt \~{}}eariasca}\}} \ and
Alberto Rodr\'iguez Casal\footnote{Departamento de Estat\'istica e Investigaci\'on Operativa, Facultade de Matem\'aticas, Universidade de Santiago de Compostela
\{\href{http://eio.usc.es/pub/alberto/}{http://eio.usc.es/pub/alberto/}\}}
}
\date{}

\maketitle

\begin{abstract}
We consider the problem of estimating the perimeter of a smooth domain in the plane based on a sample from the uniform distribution over the domain.  We study the performance of the estimator defined as the perimeter of the alpha-shape of the sample.  Some numerical experiments corroborate our theoretical findings.

\end{abstract}

\section{Introduction}

The problem of recovering topological and geometric information about the support of a distribution based on a sample has received a considerable amount of attention in a number of fields, such as computational geometry, computer vision, image analysis, clustering or pattern recognition.  This includes, for example, estimating of the number of connected components~\citep{MR2320821}, the intrinsic dimensionality~\citep{levina-bickel} and, more generally, the homology~\citep{MR2383768,MR2476414,MR2121296,MR2460371,MR1876386}, the Minkowski content~\citep{MR2341697}, as well as the perimeter and area~\citep{MR1641826,MR0169139}.
The estimation of the support or, more generally, level sets of a density is itself a rich line of research \citep{MR1345204,MR2541446,cad06, MR1447735,walther97,poincare}.  A closely related topic is that of set estimation \citep{MR1332579,cuefra10}.
We refer the reader to the classic book of \cite{MR1226450}, which treats a number of these topics.

We focus here on the problem of estimating the perimeter of the support.  Concretely, we are given a set of points $\cX_n=\{X_1,\ldots,X_n\}$, which we assume are independently
sampled uniformly at random from an unknown compact set $S \subset \bbR^2$, and our goal is to estimate the perimeter of $S$, by which we mean the length of its
boundary.  Let $\partial S$ denote the boundary of a set $S \subset \bbR^2$, namely $\partial S = \bar{S} \cap \overline{S^\comp}$, where $\bar{S}$ denotes the closure
of $S$ and $S^\comp = \bbR^2 \setminus S$ is the complement of $S$.

\subsection{Related work}
\cite{MR0169139} address this problem under the assumption that $S$ is convex and estimate its perimeter by the perimeter of the convex hull of the
sample $\cX_n$.  They obtain the precise rate of convergence in expectation, which is of order $O(n^{-2/3})$ when the boundary $\partial S$ has bounded curvature.  They also obtain an analogous result for the problem of estimating the area of $S$. \cite{MR1641826} extend their results to other sampling distributions.  See \citep{Reitzner10} for a review on more recent results on the convex hull of a random sample.

There is a series of papers that consider the problem of estimating the surface area of the boundary
of a more general class of supports $S$ but under a different sampling scheme where two samples are given, one from the uniform distribution on $S$ and
another from the uniform distribution on $G \setminus S$, where $G$ is a bounded set containing $S$.
In that line, \cite{CueFraRod07} aim at estimating the Minkowski content of $\partial S$, and introduce an estimator that is proved to be consistent under weak assumptions on the set $S$.
They obtain a convergence rate of $O(n^{-1/4})$ in dimension 2 when $\partial S$ has bounded curvature---in which case the Minkowski content coincides with the perimeter.
\cite{patrod09,PatRod08} follow their work and propose a different estimator, which is very closely related to the one we study here, obtaining an improved rate convergence of $O(n^{-1/3})$ in dimension 2.
Continuing this line of work, \cite{JimYuk11} propose an estimator of the perimeter of $S$ based on a Delaunay triangulation, which is shown to be consistent under mild assumptions on $S$.

Also closely related is the work of \cite{kim2000estimation} in the context of binary images, which includes the estimation of the length of the boundary of a horizon of the form $\{(x,y) \in [0,1]^2 : y \le g(x)\}$, where $g : [0,1] \mapsto [0,1]$ is a function with H\"older regularity.  See \secref{discussion} for further comments.

\subsection{The $r$-rolling condition}

A set $S$ is said to fulfill the $r$-rolling condition if for any
$x\in \partial S$ there is a
open ball with radius $r$, $B$, such that $B \cap S=\emptyset$ and $x\in \partial B$.
In this paper, we work under the assumption that $S$ satisfies the following condition:

\begin{quote}
{\em $S$ is a compact subset of $\bbR^2$ such that both $S$ and $S^\comp$ satisfy the $r$-rolling condition.}
\end{quote}

From a geometrical
point of view, we are assuming that a ball of radius $r$ can roll inside $S$ and $S^\comp$.
This rolling condition implies
that, for any $x\in\partial S$, there are two open balls
$B^+$ and $B^-$ such that $x\in \partial B^+\cap \partial B^-$, $B^+\subset S$ and $B^-\subset S^\comp$. In fact, it can be easily seen \citep[Lemma A.0.1]{bea-tesis}
that this is only possible if there is a (unique) unit vector $\eta_x$ (the unit normal vector at $x$ pointing outward) such
that $B^+=B(x-r\eta_x,r)$ and $B^-=B(x+r\eta_x,r)$, where $B(a,\alpha)$ denotes the open ball with radius $\alpha$ and center $a\in\mathbb{R}^2$.
See \citep{Walther99} for a comprehensive discussion, including a relation to Serra's regular model and mathematical morphology.
The $r$-rolling condition is closely linked to the notion of {\em $r$-convexity}. A set $S$ is said to be $r$-convex if for any point $x \notin \bar{S}$ there is a open ball $B$ of radius $r$ such
that $x\in B$ and $B \cap \bar{S}=\emptyset$ \citep{MR0077161,walther97}.
It is known that, if both $S$ and $S^{c}$ satisfy the $r$-rolling condition, then $S$ and $S^\comp$ are $r$-convex; see \cite[Lemma A.0.8]{bea-tesis} and also \citep{Walther99}.

The $r$-rolling condition is also connected with the idea of sets of positive reach introduced in the seminal paper \citep{MR0110078}. For a nonempty set $S \subset \bbR^2$ and $x \in \bbR^2$, define
$$
\dist(x, S) = \inf\{\|x - s\|: s \in S\},
$$
where $\|\cdot\|$ stands for the Euclidean norm.
The reach of a set $S$, denoted $\rho(S)$, is the supremum over $r > 0$ such that there is a unique point realizing $\inf\{\|x - s\|: s \in S\}$ on the set $\{x: \dist(x, S) < r\}.$
For twice differentiable submanifolds (e.g., curves), the reach bounds the radius of curvature from above~\citep[Lem.~4.17]{MR0110078}.
Also, if $S$ and $S^\comp$ satisfy the $r$-rolling condition then $\rho(\partial S)\geq r$; see \cite[Lemma A.0.6]{bea-tesis}.
Conversely, using results in \citep{cuefrapat12}, it follows easily that the converse is true if, in addition, $S$ is equal to the closure of its interior.

\subsection{The estimator}
Our estimator for the perimeter of $S$ is the perimeter of the {\em $\alpha$-shape} of $\cX_n$, for some fixed $0 < \alpha < r$.
The $\alpha$-shape of $\cX_n$ is the polygon, denoted $C_\alpha(\cX_n)$, whose edges---which we call $\alpha$-edges---are defined as follows \citep{MR713690}.
A pair $(X_{i}, X_{j})$ forms
an $\alpha$-edge if there is an open ball $B$ of radius $\alpha$ such that $X_{i}, X_{j} \in \partial B$ and $B \cap \cX_n = \emptyset$.
If $\alpha$ is large enough, the $\alpha$-shape coincides with the convex hull of the sample.  For a smaller $\alpha$, the $\alpha$-shape is not necessarily convex.  See \figref{example} for an illustration.
The $\alpha$-shape is well known in the computational geometry literature for producing good global reconstructions if the sample points are (approximately) uniformly distributed in the
set $S$.
Moreover, it can be computed efficiently in time $O(n \log n)$.
See \citep{edelsbrunner2010alpha} for a survey.

\cite{cuefrapat12} estimate the perimeter of $S$ by the outer Minskowski content of the {\em $r$-convex hull} of the sample, defined as the smallest $r$-convex set that contains
the sample.
Since the boundary of that set is smooth except at a finite number of points, the outer Minskowski coincides with the perimeter.
See \citep{ambrosio2008outer} for a broader correspondence between these two quantities.
\cite{cuefrapat12} show that this estimator is consistent, but no convergence rate is provided.  Note that, for large sample sizes, both estimators are quite similar; see \prpref{similar} for a formal statement.
From the computational point of view, the $\alpha$-shape of
the sample tends to be more stable with respect to the value of $\alpha$, and is faster to compute over a range of values of $\alpha$---the latter can be done in $O(n \log n)$ time, since the $\alpha$-shape changes a finite number of times with $\alpha$.  The $\alpha$-convex hull of the sample does not enjoy such properties.

\begin{center}
\begin{figure}[h]
\begin{minipage}{0.32\textwidth}\centering
\scalebox{0.20}{\includegraphics{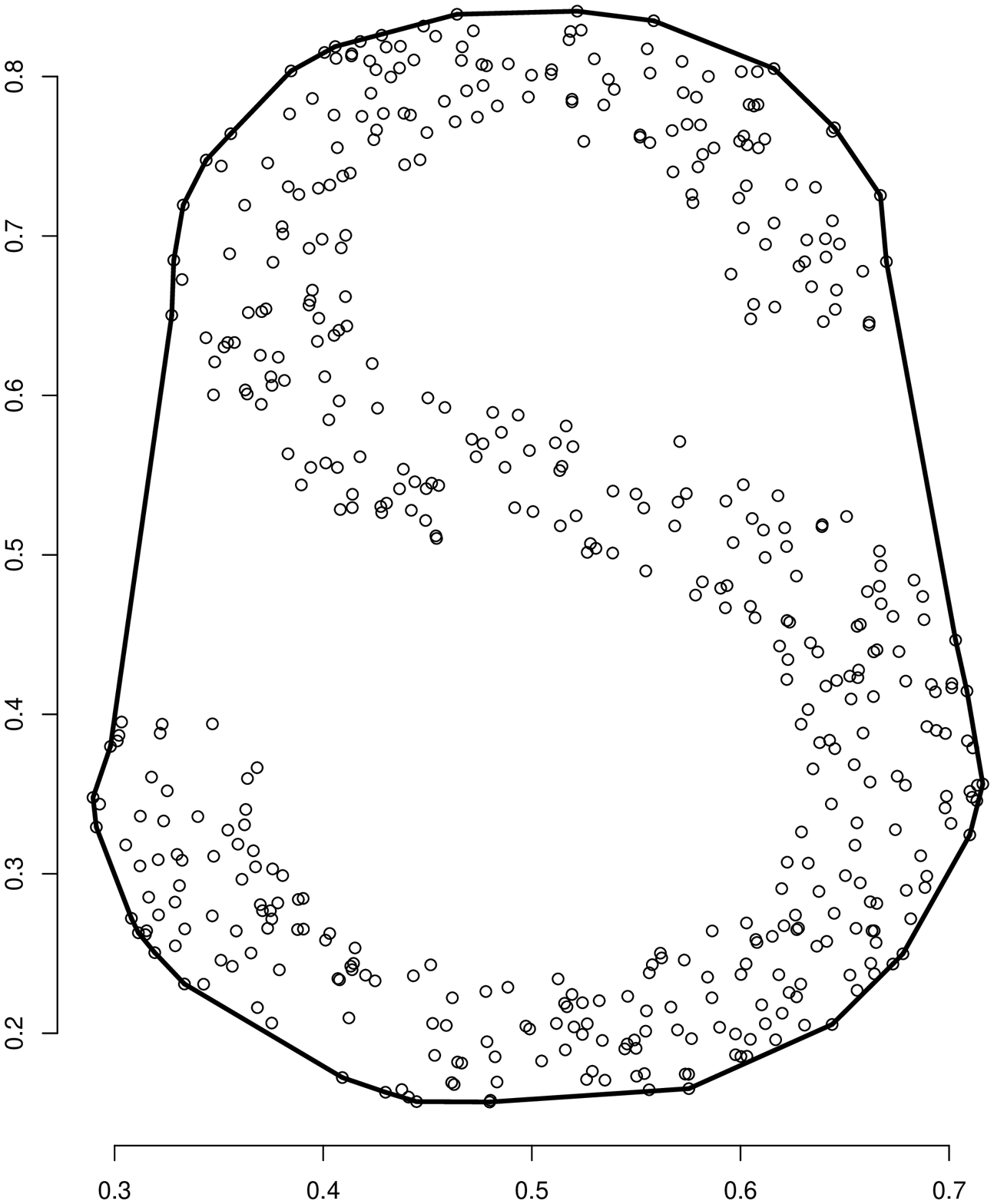}}
\end{minipage} \
\begin{minipage}{0.32\textwidth}\centering
\scalebox{0.20}{\includegraphics{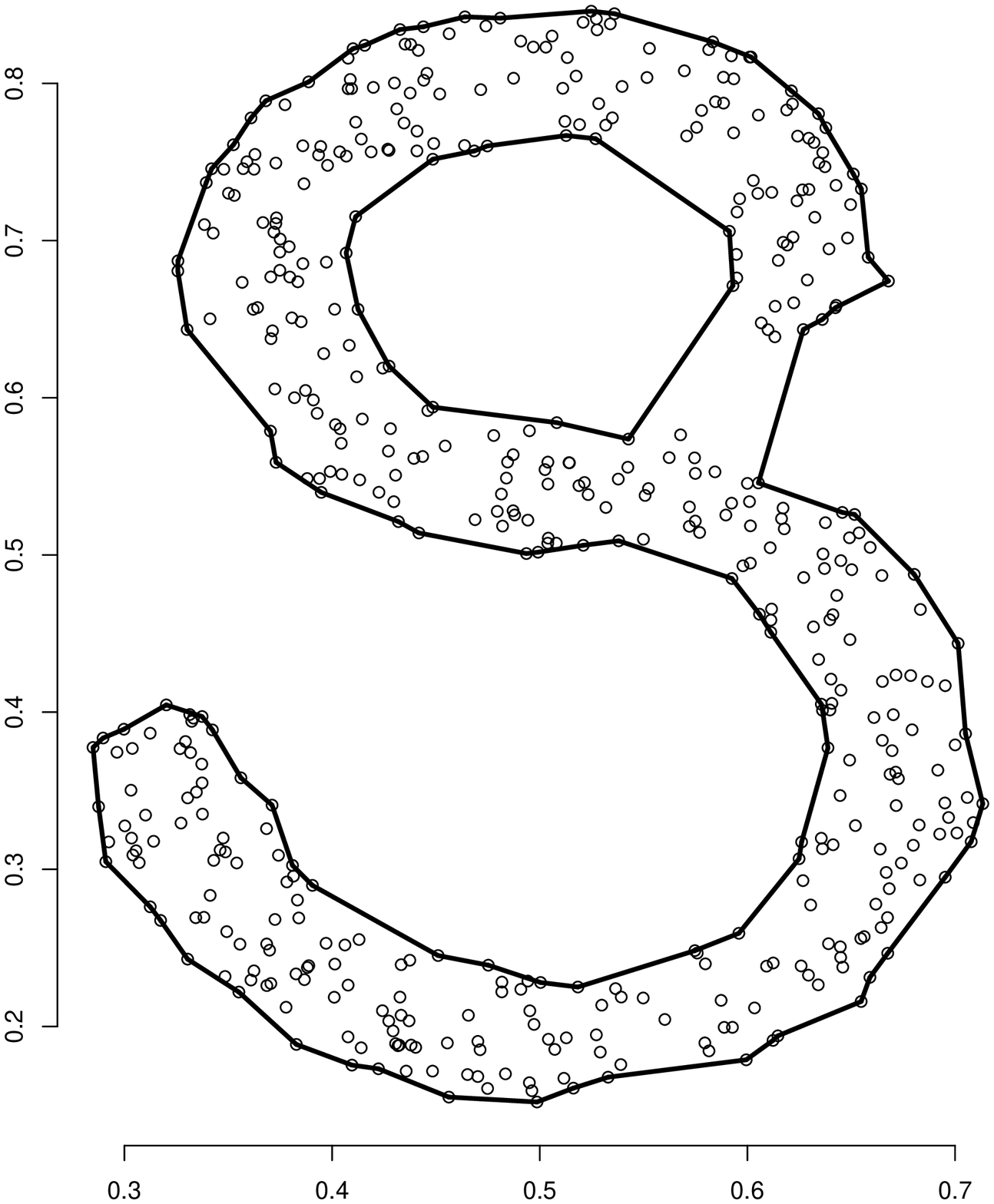}}
\end{minipage} \
\begin{minipage}{0.32\textwidth}\centering
\scalebox{0.20}{\includegraphics{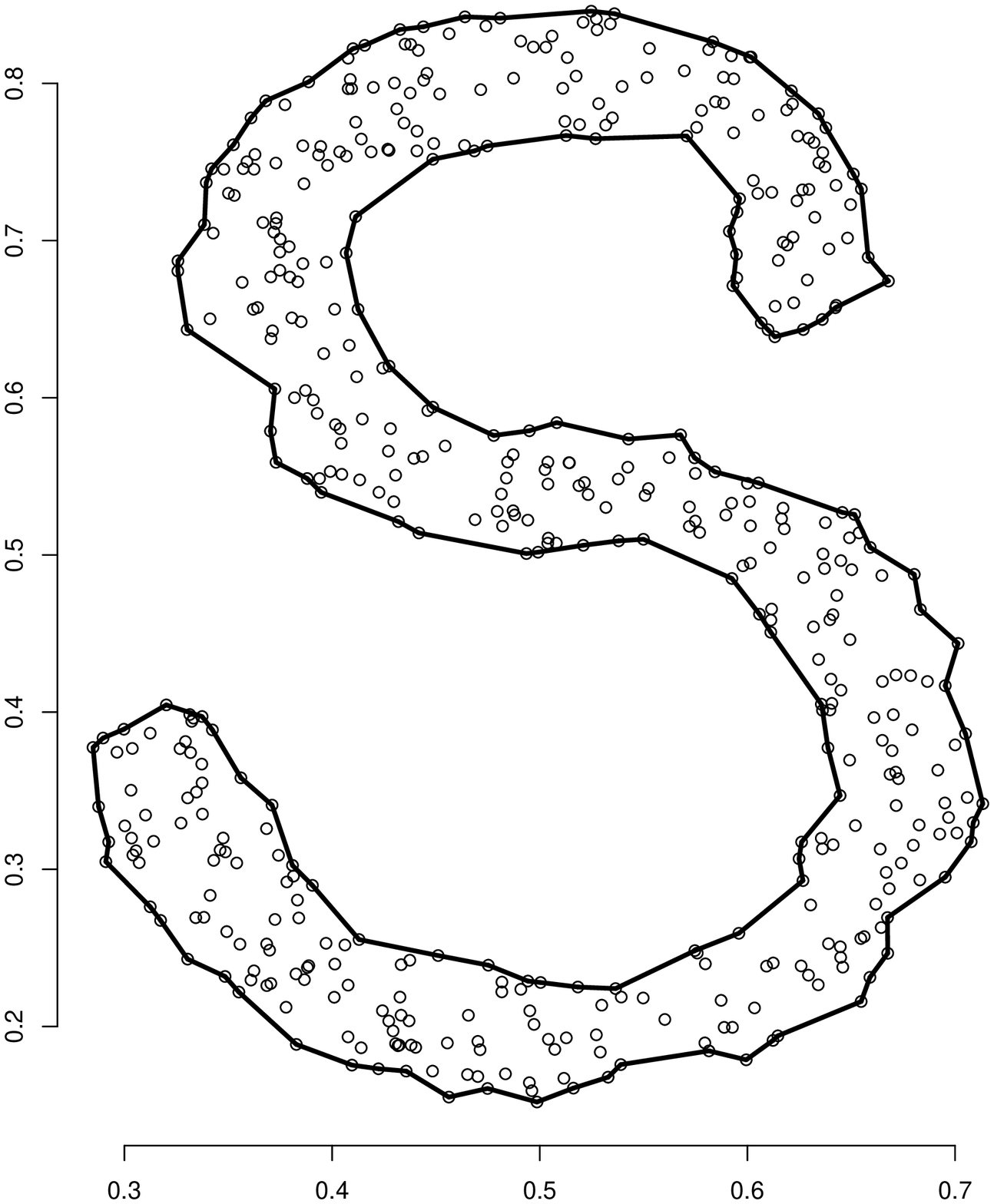}}
\end{minipage}
\label{fig:example}
\caption{The $\alpha$-shape of a sample of size $n=500$ from the uniform distribution of a thick S letter, for $\alpha=1$ (left), $\alpha=0.06$ (center) and $\alpha=0.035$ (right).
Note that in the second case the $\alpha$-shape is made of two disconnected closed curves.}
\end{figure}
\end{center}

\subsection{Main results}
Let $\lambda$ denote the one-dimensional Hausdorff measure in $\bbR^2$, normalized so that it equals 1 for a line segment of length 1, and let $\diam(A)=\sup\{\|x-y\|: \ x,y\in A\}$
denote the diameter of a set $A\subset \bbR^2$.

\begin{thm} \label{thm:main}
Let $\cX_n = (X_1,\ldots,X_n)$ be an independent sample from the uniform distribution on a compact set $S \subset \bbR^2$ such that $S$ and $S^\comp$ satisfy the $r$-rolling condition.  Fix $\alpha \in (0, r)$.  There is a
constant $A$ depending only on $(\alpha, r, \diam(S))$ and $t_0>0$ depending only on $(\alpha, r)$ such that, for all $0 \le t \le t_0$,
\begin{equation} \label{main}
\P\left(\left|\frac{\lambda(C_\alpha)}{\lambda(\partial S)}-1\right| > t\right) \leq A n^2 \exp(- n t^{3/2}/A ).
\end{equation}
\end{thm}

\begin{rem}
In particular, defining $\eps_n = (3 A \log(n)/n)^{2/3}$, with probability one,
\[(1-\eps_n) \lambda(\partial S) \le \lambda(C_\alpha(\cX_n)) \le (1+\eps_n) \lambda(\partial S),\]
eventually, by applying the Borel-Cantelli lemma. So, the convergence rate of $\lambda(C_\alpha(\cX_n))$ as an estimator of $\lambda(\partial S)$ is, up to a log factor, of order $n^{-2/3}$.  
\end{rem}

\begin{rem}
We will argue later on that the same result holds also for the perimeter of the $\alpha$-convex hull of the sample, refining, thus, the 
convergence established in \citep{cuefrapat12}.  See the discussion in \secref{discussion}.
\end{rem}

\subsection{Content}
The remaining of the paper is largely devoted to proving \thmref{main}.
In \secref{geom} we establish some auxiliary geometrical results.
\secref{edges} is dedicated to the study of $\alpha$-edges.
\thmref{main} is proved in \secref{main}.
Some numerical experiments are presented in \secref{numerical}.
We discuss some extensions and open problems in \secref{discussion}.

\subsection{Notation and preliminaries} \label{sec:prelim}
We start by introducing some notation and some general concepts.
Let $\mu(A)$ denote the Lebesgue measure of a measurable set $A \subset \bbR^2$.
For a pair of distinct points $x_1, x_2 \in \bbR^2$, let $(x_1 x_2)$ denote the line passing through $x_1$ and $x_2$, and let $[x_1 x_2]$ denote the line segment with endpoints $x_1$ and $x_2$.
For a non empty set $A \subset \bbR^2$ and $\eps > 0$, define
$$
B(A,\eps) = \{x \in \bbR^2: \dist(x, A) < \eps\}.
$$
If $A=\{x\}$ is a singleton we use the notation $B(x,\eps)$ (resp. $\bar{B}(x,\eps)$) instead of $B(\{x\},\eps)$ for denoting
the open (resp.~closed) ball of radius $\eps>0$ and center $x\in\bbR^2$. Let $P_{A}$ denote
the metric projection onto a set $A$, i.e., $P_A(x) = \argmin_{a \in A} \|x - a\|$, which is a singleton when $\dist(x,A) < \rho(A)$.
For two nonempty sets $C, D \subset \bbR^2$, let $\cH(C, D)$ denote their Hausdorff distance, defined as
\[
\cH(C, D) = \inf\{\eps > 0: C \subset B(D, \eps) \text{ and } D \subset B(C, \eps)\}.
\]
For a curve $C \subset \bbR^2$ and $x \in C$, $\vec{C}_x$ denotes the tangent subspace of $C$ at $x$ when it exists.
For two curves, $C$ and $D$, respectively differentiable almost everywhere and differentiable, and such that $\rho(D) \geq r$ and $C \subset B(D, r)$, define
the deviation angle of $C$ with respect to $D$ as
$$
\angle (C, D) = \sup_{x \in C} \angle\left( \vec{C}_x, \vec{D}_{P_D(x)} \right),
$$
where $ \angle( \vec{C}_x, \vec{D}_{P_D(x)} )\in [0,\pi/2]$ denotes the angle between the tangent spaces of $C$ and $D$ at $x$ and $P_D(x)$, respectively \citep{1479764}. Note that it is not symmetric in $C$ and $D$.

Where they appear, $\alpha$ and $r$ are fixed.
Everywhere in the proof, a constant only depends (at most) on $\alpha,r$ and the diameter of $S$.  We will leave this dependence implicit most of the time.

We let $n$ denote the sample size throughout.
We say that an event holds with high probability if it happens with probability at least $1 -Ae^{-n/A}$ for some constant $A > 0$.

\section{Some geometrical results} \label{sec:geom}

In this section we gather a few geometrical results that we will use later on in the paper.

\begin{lem} \label{lem:center-out}
Let $S\subset \bbR ^2$ such that $S$ and $S^\comp$ satisfy the $r$-rolling condition. Any ball of radius $\alpha>0$ with center in $S$ contains a ball of radius $\tfrac12 \min\{\alpha,r\}$ included in $S$.
\end{lem}

\begin{proof}
Let $\Gamma$ be a shorthand for $\partial S$.
First, we will analyze the case $\alpha\leq r$.  If $z \in S$ satisfies $\dist(z, \Gamma) \geq \alpha$, then $B(z, \alpha) \subset S$.  Now, take $z \in S$ such that $\dist(z, \Gamma) < \alpha$ and let $y$ be the
metric projection of $z$ onto $\Gamma$, which is well-defined since $\dist(z,\Gamma)<\rho(\Gamma)$.  By the $r$-rolling property, there is an open ball $B$ of radius $r$ tangent to $\Gamma$ at $y$ that
contains $z$ and $B \subset S$.
Therefore  $B(z, \alpha) \cap B$ contains the ball of radius $\alpha/2$ tangent to $\Gamma$ at $y$ that contains $z$.  See \figref{center-out} for
an illustration. This concludes the proof for $\alpha\leq r$. If $\alpha>r$, the ball of radius $\alpha$ contains the ball of radius $r$ with same center.  By what we just did, that ball contains a ball of radius $r/2$ which belongs to $S$.
%
\begin{figure}[htbp]
\centering
\begin{center}
\psset{unit=1.5cm}

\begin{pspicture}(4,3.5)
\psaxes[labels=none,ticks=none,linestyle=dashed,linewidth=0.3\pslinewidth](2,1)(0,0)(4,3.5)
\uput[180](3.5,2.12){\Large $\Gamma$}
\parabola[linewidth=2\pslinewidth](0,3)(2,1)
\uput[0](2,3){$B$}
\pscircle(2,2){1}
\pscircle(2,1.3){0.7}
\pscircle(2,1.35){0.35}
\psdot(2,1.35)
\psdot(2,1)
\uput[135](2,1.35){$z$}
\uput[225](2,1){$y$}
\end{pspicture}
\end{center}
\caption{Illustrates the proof of \lemref{center-out}.  The thick, parabolic line represents a portion of $\Gamma = \partial S$.}
\label{fig:center-out}
\end{figure}
\end{proof}

Recall that $\mu$ denotes the Lebesgue measure on $\bbR^2$.
\begin{lem} \label{lem:ball-vol}
Let $S\subset \bbR ^2$ be measurable and such that $S$ and $S^\comp$ satisfy the $r$-rolling condition. For any $\alpha\leq r$, there is a
numeric constant $A > 0$ depending only on $\alpha$ such that, for any $z \notin S$,
\[
\mu(B(z, \alpha) \cap S) \geq A \, \max(0,\alpha -\dist(z, \partial S))^{3/2}. 
\]
\end{lem}

\begin{proof}
Let $\Gamma$ be a shorthand for $\partial S$.
It suffices to consider $z \notin S$ such that $h= \alpha -\dist(z, \Gamma) > 0$.  Let $y$ be the metric projection of $z$ onto $\Gamma$, which is well-defined
since $\dist(z, \Gamma) < \alpha \leq \rho(\Gamma)$, and let $B$ be the open ball of radius $\alpha$ tangent to $\Gamma$ at $y$ and contained within $S$. It is clear
that $\mu(B(z, \alpha) \cap S) \geq \mu(B(z, \alpha) \cap B)$.  The intersection $B(z, \alpha) \cap B$ is the union of
two spherical caps symmetric with respect to line joining the two points at the intersection $\partial B(z, \alpha) \cap \partial B$.  See \figref{ball-vol} for an illustration. If $C$ denotes one of them, we therefore have $\mu(B(z, \alpha) \cap B) = 2 \mu(C)$, with $C$ a spherical cap of radius $\alpha$ and height $h$.  Its area is equal to
\[
\mu(C) = 2 \alpha^2 \int_0^{\acos(1-h/\alpha)} \sin^2(t) dt.
\]
Using the bound $\sin(t) \geq 2t/\pi$, valid for $t \in [0,\pi/2]$, and the bound $\acos(1-t) \geq \sqrt{2t}$, valid for $t \in [0,1]$, we obtain $2 \mu(C) \geq A \, h^{3/2}$ with $A = 32 \sqrt{2\alpha} / (3 \pi^2)$.
%
\begin{figure}[htbp]
\centering
\begin{center}
\psset{unit=1.5cm}

\begin{pspicture}(4,3.5)
\psaxes[labels=none,ticks=none,linestyle=dashed,linewidth=0.3\pslinewidth](2,1)(0,0)(4,3.5)
\uput[180](3.5,2.12){\Large $\Gamma$}
\parabola[linewidth=2\pslinewidth](0,3)(2,1)
\uput[0](2,1.7){$B$}
\pscircle(2,1.7){0.7}
\pscircle(2,0.7){0.7}
\psdot(2,1)
\psdot(2,0.7)
\uput[315](2,0.7){$z$}
\uput[225](2,1){$y$}
\end{pspicture}
\end{center}
\caption{Illustrates the proof of \lemref{ball-vol}.  The thick, parabolic line represents a portion of $\Gamma = \partial S$.  The intersection of the two balls is the region of interest.}
\label{fig:ball-vol}
\end{figure}
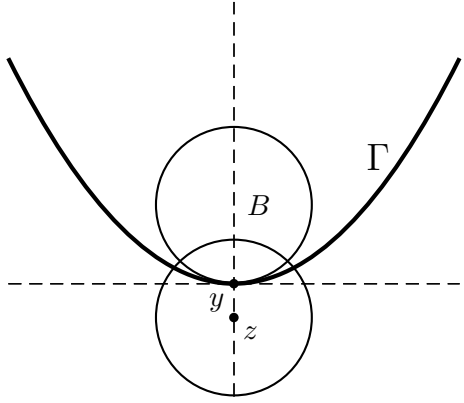
\end{proof}

For the following result, we use some heavy machinery from the seminal work of \cite{MR0110078}.
For a set $T \subset \bbR^2$, let $\cE(T)$ denote its Euler-Poincar\'e characteristic, and recall that $\lambda(T)$ denotes its length.

\begin{lem} \label{lem:length-bounded}
Suppose $S \subset \bbR^2$ is compact, with both $S$ and $S^\comp$ satisfying the $r$-rolling condition.
There are constants $A_0,A_1 > 0$ depending only on $r$ and $\diam(S)$ such that $|\cE(\partial S)|\leq A_0$ and $\lambda (\partial S) \leq A_1$.
\end{lem}

\begin{proof}
Let $\Gamma=\partial S$ and $d = \diam(S)$, and assume, without loss of generality, that $S \subset \bar{B}(0, d)$. For a given $T$
such that $\rho(T)\geq r$, let $\Phi_k$ denote the $k$th curvature measure associated with $T$, $k \in \{0,1,2\},$ as defined in \cite[Def.~5.7]{MR0110078}.
In \cite[Rem.~5.10]{MR0110078} we find that
\beq \label{proof-length-bounded1}
\sup \{|\Phi_k|(T) : T \subset \bar{B}(0, d), \, \rho(T) \ge r \} < \infty,
\eeq
where $|\Phi_k|(T)$ is the total variation of $\Phi_k$ over $T$.
Now, by \cite[Rem.~6.14]{MR0110078}, $\Phi_1(\Gamma)$ coincides with the one-dimensional Hausdorff measure, so that $|\Phi_1|(\Gamma) = \Phi_1(\Gamma) = \lambda(\Gamma)$.
From this, we deduce the existence
of $A_1$. By \cite[Th.~5.19]{MR0110078}, $\Phi_0(\Gamma)$ coincides with $\cE(\Gamma)$ and, by (\ref{proof-length-bounded1}) for $k=0$, we get that there is some constant $A_0$ such that $|\Phi_0(\Gamma)|\leq |\Phi_0|(\Gamma)\leq A_0$.
\end{proof}

We define an $\eps$-net of a set $S$ as any subset of points $x_1, \dots x_m \in S$ such that $\|x_j -x_k\| \geq \eps$ when $j \neq k$, and that, for any $x \in S$, $\|x -x_j\| < \eps$ for some $j =1, \dots, m$.  Note
that any bounded set $S \subset \bbR^2$ admits an $\eps$-net of finite cardinality.

\begin{lem} \label{lem:net}
For any bounded $S\subset \bbR^2$, there is a constant $A$ depending only on $\diam(S)$ such that, for any $0 < \eps < \diam(S)$, any $\eps$-net for $S$ has cardinality bounded by $A \eps^{-2}$.  If, in addition, both $S$ and $S^\comp$ satisfy the $r$-rolling condition, then there is a constant $A'$ depending only on $r$ and $\diam(S)$  such that any $\eps$-net for $\partial S$ has cardinality bounded by $A' \eps^{-1}$.
\end{lem}

\begin{proof}
Assume without loss of generality that $S\subset \bar{B}(0,d)$ where $d = \diam(S)$.
Let $x_1,\ldots,x_m$ be an $\eps$-net of $S$.
Since $B(x_j, \eps/2) \cap B(x_k, \eps/2)=\emptyset$ when $j \neq k$, we have
\[
\pi d^2\geq \sum_{j=1}^m \mu(\bar{B}(0,d) \cap B(x_j, \eps/2)) \geq m \pi (\eps/4)^2,
\]
using \lemref{center-out} in the last inequality.  We therefore have $m \leq 16 d^2/\eps^{2}$.  This proves the first part.

For the second part, let $\Gamma=\partial S$.
It is enough to show the results for $\eps \leq 2r$.
Note that $2r \le d$ by the $r$-rolling condition on $S$.
Let $y_1,\ldots,y_{m'}$ be an $\eps$-net of $\Gamma$. Since $B(y_j, \eps/2) \cap B(y_k, \eps/2) = \emptyset$ when $j \neq k$, we have
\beq \label{proof-net1}
m' \pi\left(\tfrac{\epsilon}{2}\right)^2 =\mu\left(\cup_{j=1}^{m'} B\left(y_j,\tfrac{\eps}{2}\right)\right)\leq \mu \left(B\left(\Gamma,\tfrac{\eps}{2}\right)\right).
\eeq
By \cite[Th.~5.6]{MR0110078}, we have
\[
\mu (B(\Gamma,\eps/2))=\eps \Phi_{1}(\Gamma)+\frac \pi 4 \eps^2\Phi_0(\Gamma),
\]
where $\Phi_1(\Gamma)=\lambda(\Gamma)$ \cite[Rem.~6.14]{MR0110078} and $\Phi_0(\Gamma)$ is the Euler-Poincar\'{e} characteristic of $\Gamma$ \cite[Th.~5.19]{MR0110078}.
By Lemma \ref{lem:length-bounded}, there are positive constants $A_0, A_1$ depending only on $r$ and $d$ such that $\lambda(\Gamma) \le A_1$ and $|\Phi_0(\Gamma)| \le A_0$, yielding
\[
\mu (B(\Gamma,\eps/2)) \le A_1 \eps + A_0 \frac\pi 4 \eps^2 \le A_2 \eps,
\]
where $A_2 = A_1 + A_0 (\pi/4) d$, using the fact that $\eps \le d$.  Plugging this into \eqref{proof-net1}, we conclude the proof of the second part.
\end{proof}

Next, we establish some basic properties of a line segment joining two points on a circle which barely intersects a set with smooth boundary.

\begin{lem} \label{lem:delta}
Let $S\subset \bbR^2$ be such that both $S$ and $S^\comp$ satisfy the $r$-rolling condition.  Fix $\alpha\in (0,r)$ and $0<t \leq \min\{\alpha, 2\alpha^2/r\}$. There
is a constant $A > 0$ depending only on $(r,\alpha)$ such
that, for any $z \notin S$ with $0<\alpha -\dist(z, S)\leq t/A$ and any $x_1, x_2 \in \partial B(z, \alpha) \cap S$, we have
\begin{eqnarray}
[x_1 x_2] & \subset & B(\partial S, t), \label{partial-dist} \\
\|x_1 -x_2\| & \leq & \sqrt{t}, \label{pair-dist} \\
\angle( [x_1 x_2], \partial S ) & \leq & \sqrt{t}. \label{angle}
\end{eqnarray}
(The angle in \eqref{angle} is well defined because of \eqref{partial-dist} and the bound $t \le \alpha <r$.)
\end{lem}

\begin{proof}
Let $\Gamma$ be a shorthand for $\partial S$.
Define $\delta= \alpha -\dist(z, S)$, and let $e_1, e_2$ denote the canonical basis vectors of $\bbR^2$.
Since $p=\dist(z, \Gamma) = \dist(z, S) =\alpha-\delta< r$, $y= P_\Gamma(z)$ is well-defined.  Without loss of generality, assume that $y$ is the
origin and that the tangent of $\Gamma$ at $y$ is the line spanned by $e_1$.  Note that
the line $(y z)$ is perpendicular to the tangent at $y$, so that $z$ is on the line defined by $e_2$ and without loss
of generality we assume $z = -p e_2$.  Let $B$ be a shorthand for $B(z, \alpha)$ and let $B^+$ (resp.~$B^-$) be the
open ball centered at $r e_2$ (resp. $-r e_2$) with radius $r$.  Since $S$ and $S^\comp$ satisfy the $r$-rolling
condition, $B^+ \subset S$ and $B^- \subset S^\comp$.  Let $x^*=\delta e_2$. By construction
$x^*$ belongs to $(y z) \cap \partial B \cap B^+$.
See \figref{delta} for an illustration.
\begin{figure}[htbp]
\label{fig:delta}
\centering
\begin{center}
\psset{unit=2cm}

\begin{pspicture}(4,4.25)
\psaxes[labels=none,ticks=none,linestyle=dashed,linewidth=0.3\pslinewidth](2,2)(0,0)(4,4)
\uput[180](3.5,3.12){\Large $\Gamma$}
\parabola[linewidth=2\pslinewidth](0,4)(2,2)
\uput[135](2,4){$B^{+}$}
\pscircle(2,3){1}
\pscircle(2,1){1}
\pscircle(2,1.7){0.7}
\psdot(2,1.7)
\psdot(2,2)
\psdot(2,2.4)
\psdot(2.216312, 2.36574)
\uput[90](2.216312, 2.36574){$x_2$}
 \psdot(1.505025,2.194975)
 \uput[90](1.505025,2.194975){$x_1$}
 \psline(1.505025,2.194975)(2.216312, 2.36574)
\uput[135](2,2.4){$x^{*}$}
\uput[225](2,1.7){$z$}
\uput[45](2,2){$y$}
\uput[45](2,0){$B^{-}$}
\end{pspicture}
\end{center} 
\caption{Illustrates the proof of \lemref{delta}.  The thick, parabolic line represents a portion of $\Gamma = \partial S$.}
\end{figure}

For any point $x \in B\cap S$,
$$
\dist(x, \Gamma) = \dist(x, S^\comp) \leq \dist(x, B^-) \leq \dist(x^*, B^-) = \delta.
$$
Direct calculations show that $\partial B \cap \partial B^-$ is given by the points $\pm a e_1 - b e_2$, where

\begin{equation*}
\left\{\begin{array}{l}a^2+(r-b)^2=r^2,\\ a^2+(p-b)^2=\alpha^2.\end{array}\right.
\end{equation*}
So, using the fact that $p = \alpha - \delta$, we have
\begin{equation}\label{bformula}
0<b=\frac{\alpha^2-p^2}{2(r-p)}=\frac{(\alpha-p)(\alpha+p)}{2(r-p)}\leq \frac{\alpha \delta}{r-\alpha}.
\end{equation}

To prove \eqref{partial-dist}, take $x \in [x_1 x_2]$.  If $x \in S$, then $x \in B \cap S$ and we saw that $\dist(x, \Gamma) \leq \delta$.  If $x \notin S$, let $C$ be the closure of the intersection of $B$ with the half-plane above the line $\bbR e_1 - b e_2$.
Since $B \cap C^\comp \subset B^-$ and $B^- \cap S = \emptyset$, necessarily $x_1, x_2 \in C$, which in turn implies that $[x_1x_2] \subset C$ since $C$ is convex.
In particular, $x \in C$, so that $\dist(x, [-a e_1, a e_1]) \le \max\{b, \delta\}$.
And since $\dist([-a e_1, a e_1], B^+) \le b$ (by symmetry), we conclude with the triangle inequality that
\beq \label{delta-proof1}
\dist(x, \Gamma) = \dist(x, S) \leq \dist(x, B^+) \leq 2 \max\{b, \delta\} \leq A_1 \delta,
\eeq
for $A_1= 2 \max\{\alpha/(r-\alpha), 1\}$. This is valid for any $x \in [x_1 x_2]$, and proves \eqref{partial-dist} for any $A \ge A_1$.

To prove \eqref{pair-dist}, we use the fact that $x_1, x_2 \in B \cap S \subset B \setminus B^-$, so that $\|x_1-x_2\|\leq \diam(B \setminus B^-)$, and $\diam(B \setminus B^-)=2a$ when $b\leq p$, which is the case since our assumptions that $\delta \le t/A$ and $t \le 2 \alpha^2/r$ imply $\delta \le (r-\alpha) \alpha/r$, which forces $b \le p$ by \eqref{bformula}.
Continuing, we then have
$$
a^2=r^2-(r-b)^2=b(2r-b)\leq 2 b r \le A_1 r \delta,
$$
by \eqref{delta-proof1}.
From this we get
\beq \label{A2}
\|x_1-x_2\|\leq \diam(B \setminus B^-) = 2a\leq \sqrt{A_2\delta},
\eeq
where $A_2= 4 A_1 r$.
This proves \eqref{pair-dist} for any $A \ge \max\{A_1,A_2\}$.

We turn to proving \eqref{angle}.
We first note that $\angle( [x_1 x_2], \Gamma )$ is well-defined. Indeed, by assumption $\delta\leq t/A$, with $A \ge A_1 \ge 1$, and $t \leq \alpha$, so that $B([x_1x_2], \Gamma) \le \alpha$ by \eqref{partial-dist}, and we conclude with the fact that $\rho(\Gamma)\geq r>\alpha$.
For any $x\in [x_1 x_2]$ we can therefore compute the point $y'=P_{\Gamma}(x)$. Using the triangle inequality for angles, we have
\begin{equation}\label{bound-angle}
\angle( [x_1 x_2], \vec{\Gamma}_{y'} )\leq \angle( [x_1 x_2], \vec{\Gamma}_{y} )+\angle( \vec{\Gamma}_{y},\vec{\Gamma}_{y'}   )=\theta_1+\theta_2.
\end{equation}
We first bound $\theta_1$.  Direct trigonometric calculations show that
$$
\sin(\theta_1)\leq \frac{a}{\alpha}\leq \frac{\sqrt{A_2\delta}}{2\alpha},
$$
where the last inequality comes from \eqref{A2}.
We use the fact that $\sin(\theta) \ge 2 \theta/\pi$ for all $\theta \in [0,\pi/2]$, we get $\theta_1\le A_3\sqrt{\delta}$, where $A_3=\pi \sqrt{A_2}/(4\alpha)$.
It remains to bound $\theta_2$ in \eqref{bound-angle}. We have $y=P_{\Gamma}(x^*)$ and $y'=P_{\Gamma}(x)$, and $\dist(x^*,\Gamma)=\delta<\alpha$ by construction, and
also $\dist(x,\Gamma)\leq t\leq \alpha$ because of \eqref{partial-dist}.  Hence, by \cite[Th.~4.8(8)]{MR0110078}, we get
$$
\|y-y'\|\leq \frac{r}{r-\alpha}\|x-x^*\|.
$$
Using the fact that $x, x^* \in B \setminus B^-$, and then \eqref{A2}, we have
$
\|x-x^*\|\leq \sqrt{A_{2}\delta}.
$
Now, if we denote by $\vec{\eta}_y$ and $\vec{\eta}_{y'}$ the outward
pointing unit normal vector of $\Gamma$ at $y$ and $y'$ respectively, \cite[Th.~1]{walther97} ensures that
$$
\|\vec{\eta}_y-\vec{\eta}_{y'}\|\leq \frac{1}{r}\|y-y'\|.
$$
Since $\<\vec{\eta}_{y}, \vec{\eta}_{y'}\> = \<\vec{\Gamma}_{y}, \vec{\Gamma}_{y'}\> = \cos \theta_2$, we get
$$
\|\vec{\eta}_y-\vec{\eta}_{y'}\| = \sqrt{2 -2\cos \theta_2} = 2 \sin (\theta_2/2).
$$
We arrive at
$$\sin (\theta_2/2) \leq \frac{\sqrt{A_2\delta}}{2 (r-\alpha)}.$$
As before, this implies that $\theta_2 \le A_4\sqrt{\delta}$, where $A_4=\pi \sqrt{A_2}/(4(r-\alpha))$.
We conclude that
$$
\angle( [x_1 x_2], \vec{\Gamma}_{y'} )\leq (A_3+A_4)\sqrt{\delta}=\sqrt{A_5\delta},
$$
which proves \eqref{angle} for any $A \ge \max\{A_1,A_2,A_5\}$.
\end{proof}

The following is a technical result involving two line segments, one on each of two intersecting circles of same radius, and a line passing through these line segments.
\begin{lem} \label{lem:two-edges}
Let $x_0, x_0' \in \bbR^2$ such that $0 < \|x_0 - x_0'\| < 2 \alpha$, and
let $x_1, x_2 \in \partial B(x_0, \alpha) \setminus B(x'_0, \alpha)$ and $x_1', x_2' \in \partial B(x'_0, \alpha) \setminus B(x_0, \alpha)$.  Let $L$ be any line intersecting both $[x_1 x_2]$ and $[x_1' x_2']$.
Then there is a constant $A > 0$ depending only on $\alpha$ such that
$$
\max\Big\{\angle ((x_1 x_2), L), \, \angle ((x_1' x_2'), L)\Big\}\leq A \Big(\|x_0 - x'_0\| + \max_{i,j \in\{1,2\}} \|x_i - x_j'\| \Big).
$$
\end{lem}

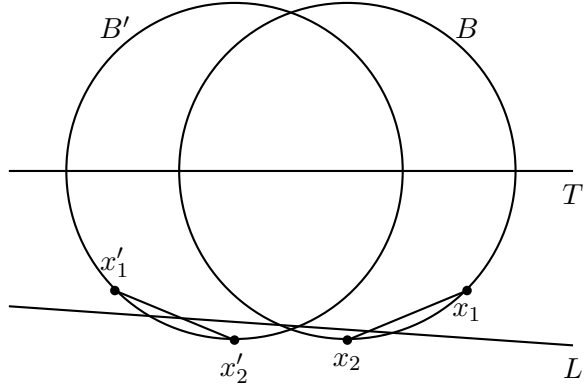
\begin{figure}[ht]
\centering
\begin{center}
\psset{unit=1.5cm}

\begin{pspicture}(5,3.5)
\pscircle(2,2){1.5}
\pscircle(3,2){1.5}
\uput[90](0.9393398,3.06066){$B'$}
\uput[90](4.06066,3.06066){$B$}
\psline(0,2)(5,2)
\psdot(0.9393398,0.9393398)
\psdot(2,0.5)
\uput[270](5,2){$T$}
\psdot(3,0.5)
\uput[270](3,0.5){$x_2$}
\psdot(4.06066,0.9393398)
\uput[270](4.06066,0.9393398){$x_1$}
\psline(3,0.5)(4.06066,0.9393398)
\uput[90](0.9393398,0.9393398){$x_1'$}
\uput[270](2,0.5){$x_2'$}
\psline(0.9393398,0.9393398)(2,0.5)
\psline (0,0.8)(5,0.4532108)
\uput[270](5,0.4532108){$L$}
\end{pspicture}
\end{center}
\caption{Illustrating the proof of \lemref{two-edges}.}
\label{fig:two-edges}
\end{figure}

\begin{proof}
Let $B$ and $B'$ be a shorthand for $B(x_0, \alpha)$ and $B(x'_0, \alpha)$, respectively.
Since the maximum above is bounded by $\pi/2$, it is enough to prove the inequality when
$$a = \|x_0 - x'_0\| + \max_{i,j \in\{1,2\}} \|x_i - x_j'\| < \alpha.$$
Let $T = (x_0 x'_0)$, and let $H$ and $\tilde{H}$ denote the two half-spaces defined by $T$.  Let $t$ denote the
intersection point $(\partial B\setminus B') \cap T$, and define $t'$ analogously. Let $m$ denote the intersection point $\partial B \cap \partial B' \cap H$,  and define $\tilde{m}$ analogously.
See \figref{two-edges} for an illustration.

We claim that, when $a < \alpha$, the points $x_1, x_2, x_1', x_2'$ are either all in $H$ or all in $\tilde{H}$.  Indeed, when $x_i$ and $x_j'$ are on opposite sides of $T$, then either $x_i \in \overarc{m t}$ and $x'_j \in \overarc{\tilde{m} t'}$, or $x_i \in \overarc{\tilde{m} t}$ and $x'_j \in \overarc{m t'}$.  (For two points $s,t \in \partial B$, $\overarc{st}$ denotes the shorter arc defined on $\partial B$ by $s$ and $t$.)
The distance between a point in $\overarc{m t}$ and a point in $\overarc{\tilde{m} t'}$ is not smaller than the minimum of $\|t - \tilde{m}\| \ge \sqrt{2} \alpha$ and $\|m - \tilde{m}\| \ge \sqrt{3} \alpha$, since $0 < \|x_0 - x_0'\| < \alpha$.
Therefore, assume without loss of generality that $x_1, x_2, x_1', x_2' \in H$.

Let $y$ be the point in $H \cap \partial B$ furthest from $T$, so the tangent of $\partial B$ at $y$ is parallel to $T$.  Define $y'$ similarly, with $B'$ in place of $B$.
We claim that $x_1, x_2 \in B(y, \sqrt{2}a)$ and $x_1', x_2' \in B(y', \sqrt{2}a)$.
We prove this for $x_1$, without loss of generality, and consider the two possible cases:
\bitem
\item If $x_1 \in \overarc{y m}$, then
$$\|y -x_1\| \leq \|y -m\| \le \|y -y'\| = \|x_0 - x'_0\| \leq a.$$
\item If $x_1 \in \overarc{t y}$, let us define $h=\|x_1-y\|$, $d=\dist(x_1,(yx_0))$ and $z=P_{(yx_0)}(x_1)$. By the Pythagoras theorem,
\begin{alignat*}{2}
d^2 &+ \|y-z\|^2 &&= h^2, \\
d^2 &+ (\alpha-\|y-z\|)^2 &&= \alpha^2.
\end{alignat*}
%
From this we get $d^2=h^2(1-h^2/(4\alpha^2))\geq h^2/2$, where the inequality is due to $s \leq \dist(t, y) = \sqrt{2}\alpha$.
But $d \leq \max_{i,j \in\{1,2\}} \|x_i - x_j'\|\leq a$. Hence, $h\leq \sqrt{2}a$, as claimed.
\eitem



By the fact that $B$ is convex, the angle between $(x_1 x_2)$ and $T$ is bounded
from above by the maximum angle between $T$ and the tangent of $\partial B$ at any point in $\overarc{x_1 x_2}$.
Moreover, by direct calculations, similar to that on \lemref{delta}, for any point on $x \in \partial B$ such that $\|y -x\| \leq \sqrt{2} \alpha$, the
angle between $T$ and the tangent of $\partial B$ at $x$ is bounded by $2 \asin(\|y - x\|/(2 \alpha)) \le \pi \|y -x\|/(2\alpha)$.
Hence, by the fact that $x_1, x_2 \in B(y, \sqrt{2} a) \subset B(y, \sqrt{2} \alpha)$, we have
$$\angle ((x_1 x_2), T) \leq \frac\pi{2\alpha} \max\{\|y -x_1\|, \|y -x_2\|\} \leq \frac\pi{2\alpha} \sqrt{2} a = \frac{\pi a}{\sqrt{2} \alpha}.$$
Similarly,
$$\angle ((x_1' x_2'), T)\leq  \frac{\pi a}{\sqrt{2} \alpha}.$$
%

By an analogous convexity argument, coupled with the fact that all the action is in half-space $H$, $\angle(L, T)$ is bounded from
above by the maximum of any angle between $T$ and a tangent of $\partial B$ at any point in $\overarc{x_1 x_2}$, or any angle
between $T$ and a tangent of $\partial B'$ at any point in $\overarc{x_1' x_2'}$.  Hence, as before, we get
\[
\angle(L, T) \le \frac{\pi a}{\sqrt{2} \alpha}.
\]

All the bounds combined, together with the triangle inequality, yield
$$
\angle ((x_1 x_2), L) \leq \angle ((x_1 x_2), T) + \angle (T, L) \leq \frac{2 \pi a}{\sqrt{2} \alpha},
$$
and similarly for $(x_1' x_2')$.
\end{proof} \medskip

The following result is useful when comparing the length of two curves in terms of their Hausdorff distance and their deviation angle.

\begin{lem}[Th.~43 in~\citep{1479764}]
\label{lem:main-approx}
Let $\Gamma$ be a compact curve in $\bbR^2$ such that $\rho(\Gamma) \geq r$ and let $C$ be another curve in $\bbR^2$, differentiable almost everywhere, such that $C \subset B(\Gamma, r)$ and $P_\Gamma$ is one-to-one on $C$.  Then
$$
\frac{\cos \angle (C, \Gamma)}{1 + \frac1r \cH(C, \Gamma)} \leq \frac{\lambda(\Gamma)}{\lambda(C)} \leq \frac{1}{1 - \frac1r \cH(C, \Gamma)}.
$$
\end{lem}

\begin{proof}
The result is an immediate consequence of~\cite[Th.~43]{1479764} and the fact that the reach bounds the radius of curvature from above \citep[Lem.~4.17]{MR0110078}.
\end{proof}

\section{Some properties of $\alpha$-edges}\label{sec:edges}

Our standing assumption in this section is the following:
\renewcommand{\theenumi}{($\star$)}
\renewcommand{\labelenumi}{\theenumi}
\benum
\item \label{setting} The data points $\cX_n=\{X_1,\ldots,X_n\}$ are independently sampled from a uniformly distribution with compact support $S \subset \bbR^2$ such that both $S$ and $S^\comp$ satisfy the $r$-rolling condition.
\eenum
\renewcommand{\theenumi}{\arabic{enumi}.}
\renewcommand{\labelenumi}{\theenumi}

For any pair of distinct data points within distance $2\alpha$ from each other, there are only
two circles of radius $\alpha$ passing through them, symmetric with respect to the line joining the two points.  In the special case
of an $\alpha$-edge, at least one of the two circles is empty of data points inside.  The following result implies that, with probability tending to one, the center of such a circle lies outside of $S$.

\begin{prp} \label{prp:center-out}
Assume {\em \ref{setting}}.
For any $\alpha>0$, there is a constant $A > 0$ depending only on $(\alpha,r,\diam(S))$ such that, with probability at least $1 - A e^{- n/A}$, there are no open balls of radius $\alpha$ with center in $S$ empty of data points.
\end{prp}

\begin{proof}
Let $d = \diam(S)$ and assume without loss of generality that $S \subset \bar{B}(0, d)$.
We will focus on the case $\alpha\leq r$. The case $\alpha>r$ can be analyzed similarly. By \lemref{center-out}, if there is a ball of radius $\alpha$ with center in $S$ empty of data points, then there is a ball
of radius $\alpha/2$ included within $S$ that is empty of data points.  By  \lemref{net}, there is an $(\alpha/5)$-net of $S$, denoted $z_1, \dots, z_m$, satisfying $m \le A_1$, where $A_1$ depends only on $d$ and $\alpha$.  By the triangle inequality any ball of radius $\alpha/2$ included within $S$ contains a ball of the form $B(z_k, \alpha/5)$.  Hence,
\beqn
\P\big(\exists z \in S : \cX_n \cap B(z, \alpha) = \emptyset \big)
&\le& \P\big(\exists k =1,\dots, m : \cX_n \cap B(z_k, \alpha/5) = \emptyset \big) \\
&\le& \sum_{k=1}^m \P\big(\cX_n \cap B(z_k, \alpha/5) = \emptyset \big) \\
&=& \sum_{k=1}^m \left[1 - \frac{\mu(B(z_k, \alpha/5))}{\mu(S)}\right]^n \\
&\le& A_1 \left[1 - (\alpha/(5d))^2\right]^n,
\eeqn
where in the second inequality we used the union bound and in the third we used the fact that $m \le A_1$ and $S \subset \bar{B}(0,d)$.
Therefore the result holds with $A=\max\{A_1, -1/\log[1 - (\alpha/(5d))^2]\}$.
\end{proof}

\begin{rem}\label{rem:alpha-extremes}
We say that a data point is $\alpha$-isolated if there are no other data points within distance $2 \alpha$ from it.
Suppose that $X_i$ is $\alpha$-isolated so that $B(X_i, 2 \alpha) \cap \cX_n = \{X_i\}$.
By the $r$-convexity of $S^\comp$, there is an open ball $B \subset S$ with radius $\alpha$ such that $X_i \in B$, which in particular satisfies $B\subset B(X_i,2\alpha) \cap S$.
Let $B' \subset B$ be an open ball of radius $\alpha/2$ such that $X_i\notin B'$.  By construction, $B'$ is included within $S$ and is empty of data points.  We conclude by \prpref{center-out} that, under \ref{setting}, with high probability, there are no $\alpha$-isolated data points.
\end{rem}

\begin{prp} \label{prp:similar}
Take $\alpha > 0$ and finite set of points $\cX \subset \bbR^2$ such that there are no $\alpha$-isolated points.  Then the vertices of the $\alpha$-shape of $\cX$ and the vertices of the $\alpha$-convex hull of $\cX$ coincide.
\end{prp}

\begin{proof}
Let $C$ and $H$ denote the $\alpha$-shape of $\cX$ and the $\alpha$-convex hull of $\cX$, respectively.
Note in particular that $H = \bigcap_{B \in \cB} B^\comp$ where $\cB$ is the set of open balls of radius $\alpha$ that do not intersect $\cX$.
First, take $x \in \cX$ such that $x\in\partial H$.  By \cite[Prop.~2]{cuefrapat12}, there is a open ball $B$ of radius $\alpha$ such that $x \in \partial B$ but $B\cap \cX=\emptyset$. 
Let $B$ pivot on $x$.  Since $x$ is not $\alpha$-isolated, the ball will eventually hit another data point, denoted $x'$.   Then $x$ and $x'$ belong to the
boundary of an open ball $B'$ of radius $\alpha$ that does not contain any other data point by construction---for otherwise the ball would have hit that another data point before $x'$---so $[xx']$ forms an $\alpha$-edge.  
This implies that $x$ is a vertex of $C$.
By definition of $H$ above, $B' \subset H^\comp$.
Therefore $x \in \overline{B'} \subset \overline{H^\comp}$, and since $x \in H$, we have $x \in H \cap \overline{H^\comp} = \partial H$.
\end{proof}

The next proposition bounds the expected number of $\alpha$-edges.

\begin{prp} \label{prp:number}
Assume {\em \ref{setting}}.
For  any $\alpha \in (0,r)$, there is a constant $A > 0$ depending only on $(\alpha,r,\diam(S))$
such that the expected number of $\alpha$-edges is bounded by $A n^{1/3}$.
\end{prp}

\begin{proof}
Let $N_\alpha^{\rm shape}$ and $N_\alpha^{\rm h ull}$ denote the number of vertices
of the $\alpha$-shape and $\alpha$-convex hull, respectively, and let $F$ denote the event that there are no $\alpha$-isolated points.
By \prpref{similar}, $N_\alpha^{\rm shape} = N_\alpha^{\rm h ull}$ on $F$, so that $N_\alpha^{\rm shape} \le N_\alpha^{\rm h ull} \mathbbm{1}_F + n \mathbbm{1}_{F^\comp}$, and consequently
\[
\E(N_\alpha^{\rm shape}) \le \E(N_\alpha^{\rm hull}) + n  \P(F^\comp).\]
On the one hand, $\P(F^\comp) = 1 - \P(F) \le A_1 e^{-n/A_1}$ for some constant $A_1$, by \prpref{center-out} and \remref{alpha-extremes}.
On the other hand, by \citep[Th.~3]{patrod13}, $\E(N_\alpha^{\rm hull}) \le A_2 n^{1/3}$, for some constant $A_2$.
From this, we conclude.
\end{proof}

\begin{rem} \label{rem:prob}
For $i < j$, let $G_{ij}$ be the event that $[X_iX_j]$ forms an $\alpha$-edge.  By the fact that the points are iid, $\P(G_{ij})$ is independent of $i < j$.  Hence, the expected number of $\alpha$-edges $\binom{n}{2} \P(G_{ij})$ and \prpref{number} implies that $\P(G_{ij}) \le A n^{-5/3}$ for some constant $A$.
\end{rem}

The next result ensures that, with high probability, for each connected component of $\partial S$ there is at least one $\alpha$-edge within distance $\alpha$.

\begin{prp} \label{prp:exists}
Assume {\em \ref{setting}}.
For any $\alpha\in (0,r)$, there is a constant $A > 0$ depending only on $(\alpha,r,\diam(S))$  such that, with probability at least $1-Ae^{-n/A}$, for any connected component of $\partial S$, there is an $\alpha$-edge with an endpoint within distance $\alpha$ of that component.
\end{prp}

\begin{proof}
Suppose that all the open balls of
radius $\alpha/2$ centered at a point in $S$ intersect the sample. By Proposition \ref{prp:center-out} this happens
with probability at least $1 - A e^{-n/A}$ for
some constant $A > 0$.
We saw in \remref{alpha-extremes} that this implies that there are no $\alpha$-isolated
data points. Let $\Gamma_k$ be a connected component of $\Gamma=\partial S$.
Fix $y \in\Gamma_k$ and let $\eta$ denote the normal
unit vector of $\Gamma_k$ at $y$ pointing away from $S$.
For $s \ge 0$, define $y_s = y + s \eta$ and let $s^* = \inf\{s > 0: B(y_s, \alpha) \cap \cX_n = \emptyset\}$.  Notice that $B(y_\alpha,\alpha)\subset S^\comp$ and, therefore, it
is empty of data points.  Hence, $s^*<\alpha$.  Moreover, we also have $s^* > 0$, since we are assuming that $B(y_{0},\alpha/2)$ contains at least one data point (since $y_0 = y \in S$).
By construction, there exists a data point $X_i \in \partial B(y_{s^*},\alpha)$.
Now, pivot the ball $B(y_{s^*},\alpha)$ on $X_{i}$ as we did in the proof of \prpref{similar}.
Since $X_i$ is not $\alpha$-isolated, the ball will eventually hit another data point, denoted $X_j$, and $[X_i X_j]$ will form an $\alpha$-edge.  And, since $\|X_i - y_{s^*}\| = \alpha$ and $y_{s^*} \in S^\comp$ (remember $0<s^*<\alpha$), there
is $z\in [X_i y_{s^*}]$ such that $z\in \Gamma$.
We now use the fact that $B(y_{s^*},\alpha)\cap \Gamma$ is contractible \citep[Rem.~4.15]{MR0110078}, and since $B(y_s^*,\alpha)\cap\Gamma_k\neq \emptyset$, we must have $B(y_{s^*},\alpha)\cap\Gamma=B(y_{s^*},\alpha)\cap\Gamma_k$, which in turn implies that $z\in\Gamma_k$ and, therefore, $\dist(X_i,\Gamma_k)<\alpha$.
\end{proof}

Next, we prove some quantitative results about $\alpha$-edges.  In plain English, we show that, with probability tending to one, $\alpha$-edges are near the boundary
of $S$, have small length and their deviation angle with the boundary of $S$ is small.
%

\begin{prp} \label{prp:edge}
Assume {\em \ref{setting}}.
For $i < j$, let $G_{ij}$ denote the event that $[X_{i} X_{j}]$ is an $\alpha$-edge, and for $t > 0$, let $H_{ij, t}$ denote the event that
\beq \label{edge}
[X_{i} X_{j}] \subset B(\partial S, t), \quad
\|X_{i} -X_{j}\| \leq \sqrt{t} \quad \text{and} \quad
\angle([X_{i} X_{j}], \partial S) \leq \sqrt{t}.
\eeq
For any $\alpha\in (0,r)$, there is a constant $A > 0$ depending only on $(\alpha,r,\diam(S))$ such that, for any $0<t \leq \min\{\alpha, 2\alpha^2/r\}$, $\P(G_{ij} \cap H_{ij, t}^\comp) \le A e^{-n t^{3/2}/A}$.
\end{prp}

\begin{proof}
Let $\Gamma$ be a shorthand for $\partial S$.
For any two distinct points $x, x' \in \bbR^2$ such that $\|x -x'\| < 2 \alpha$, define
\[\zeta^\pm(x, x') =x+ \alpha \, \Xi_{\pm \theta}\left(\frac{x'-x}{\|x'-x\|}\right),\]
where $\theta = \acos(\|x -x'\|/(2 \alpha))$ and $\Xi_\theta$ denotes the rotation at angle $\theta$. By
construction, $x,x' \in \partial B(\zeta^\pm(x, x'), \alpha)$, and $\zeta^\pm(x, x')$ are the only two points with this property.  Let $\zeta^\pm_{ij}$ be short for $\zeta^\pm(X_i, X_j)$, if $\|X_i-X_j\|<2\alpha$, and $(\zeta^+_{ij}, \zeta^-_{ij}) = (X_i, X_j)$, otherwise.

Let $E$ be the event that there are no open balls of radius $\alpha$ with center in $S$ empty of data points.  We studied this event in \prpref{center-out}.
With $A_1$ denoting the constant of \lemref{delta}, we have
\begin{multline*}
H_{ij, t}^\comp \cap G_{ij} \cap E \subset \Big\{\exists \, \eps \in \{-,+\} :
\cX_n \cap B(\zeta_{ij}^\eps, \alpha) = \emptyset, \zeta_{ij}^\eps\notin S \text{ and} \dist(\zeta_{ij}^\eps, S) < \alpha - t/A_1\Big\}.
\end{multline*}
Therefore, the union bound gives
\[\P(H_{ij, t}^\comp \cap G_{ij} \cap E) \le \sum_{\eps = \pm} \pr{\cX_n \cap B(\zeta_{ij}^\eps, \alpha) = \emptyset ,\zeta_{ij}^\eps\notin S  \text{ and} \dist(\zeta_{ij}^\eps, S) < \alpha - t/A_1}.\]
With $A_2$ denoting the constant of \lemref{ball-vol}, for any deterministic point $\zeta \notin S$ such that $\dist(\zeta, S) < \alpha - t/A_1$, we have
\beqn
\pr{\cX_{n-2} \cap B(\zeta, \alpha) = \emptyset}
&=& \left(1 -\frac{\mu(S \cap B(\zeta, \alpha))}{\mu(S)}\right)^{n-2} \\
&\le& \left(1 -\frac{A_2 t^{3/2}}{A_1^{3/2}\pi d^2}\right)^{n-2} \\
&\leq& A_3 e^{- n t^{3/2}/A_3},
\eeqn
for some constant $A_3$ which depends only on $\alpha, r$ and $d := \diam(S)$.
Hence, conditioning on $(X_i, X_j)$, we have
\beqn
\pr{\cX_n \cap B(\zeta_{ij}^\eps, \alpha) = \emptyset, \ \zeta_{ij}^\eps\notin S \text{ and} \dist(\zeta_{ij}^\eps, S) < \alpha - t/A_1}
\le A_3 e^{- n t^{3/2}/A_3}.
\eeqn
Together with \prpref{center-out}, we arrive at
\[\P(H_{ij, t}^\comp \cap G_{ij}) \le \P(H_{ij, t}^\comp \cap G_{ij} \cap E)+\P(E^\comp)\le A_4 e^{- n t^{3/2}/A_4},\]
for some constant $A_4$, again depending only on $(\alpha,r,d)$.
\end{proof}

The next two results combined imply that, with high probability, the $\alpha$-edges form a
simple polygon in one-to-one correspondence with $\partial S$.  The first result shows
that, with high probability, two distinct points in the union of all $\alpha$-edges do not project on
the same point on $\partial S$. We also show that $\alpha$-edges are all one-sided in
the sense that at least one of the two open balls of radius $\alpha$ that circumscribes an $\alpha$-edge contains a data point.

\begin{prp} \label{prp:injective}
Assume {\em \ref{setting}}.
For any $\alpha\in (0,r)$, there is a constant $A > 0$ depending
only on $(\alpha,r,\diam(S))$ such that, with probability at least $1 - A e^{-n/A}$: (i) all $\alpha$-edges are one-sided; and (ii) the metric projection onto $\partial S$ is injective on the union of all $\alpha$-edges.
\end{prp}

\begin{proof}
Let $\Gamma$ be a shorthand for $\partial S$ and $d =\diam(S)$.
Assume there are no balls of radius $\alpha$ with center in $S$ empty of data points and that, for $t$ fixed (and chosen small enough in what follows), all
the $\alpha$-edges satisfy \eqref{edge}.  Both events happen together with probability at least $1 - A e^{-n/A}$, for some constant $A > 0$, by Propositions~\ref{prp:center-out} and~\ref{prp:edge}.

We first show that, if $t$ is small enough, all $\alpha$-edges are one-sided. Let $[x_1x_2]$ ($x_1=X_{i_1},x_2=X_{i_2}$) be an arbitrary $\alpha$-edge.  Let $x_m=(x_1+x_2)/2$ be the
midpoint of that $\alpha$-edge and $\rho=(\alpha^2-\|x_1-x_m\|^{2})^{1/2}$. If there is a ball of radius $\alpha$, $B$, such that $x_1, x_2\in \partial B$, then
the center of $B$ is either $z_e=x_m+\rho u$ or $ z_{s}=x_m-\rho u$, where $u$ is the unit vector orthogonal to $(x_1x_2)$ such that $\langle u,\eta\rangle >0$, $\eta$ being
the outward pointing unit normal vector at $y_m=P_{\Gamma}(x_m)$, which is well-defined when $t<r$.
Notice that the vector $u$ is well defined when $\sqrt{t} < \pi/2$,
since in that case $(x_1x_2)$ is not orthogonal to $\Gamma$. We will prove that, for $t$ even smaller, $z_s\in S$ and therefore $B(z_s, \alpha)$ is not empty of sample points.
Define $c_{s}=y_m-\rho \eta$ and $c=y_m-r\eta$.  By the $r$-rolling property, $B(c,r)\subset S$.
By the triangle inequality and \eqref{edge}, we have
$$
\|z_{s}-c_{s}\|\leq \|x_m-y_m\|+\rho\|u-\eta\| \leq t +\alpha\|u-\eta\|,
$$
with, for $t$ small enough,
$$
\|u-\eta\|^2 = 2(1-\langle u,\eta\rangle ) \le 2(1 - \cos \angle([x_{1} x_{2}], \Gamma) ) \leq 2t,
$$
using \eqref{edge} (i.e., $\angle([x_{1} x_{2}], \Gamma) \leq \sqrt{t}$) and the fact that $\cos(a)\geq 1-a^2$ for any $a \in \bbR$.
Using the triangle inequality and \eqref{edge}, again, we get
$$
\|z_s-c\|\leq \|z_s-c_s\|+\|c_s-c\|\leq t + \alpha \sqrt{2 t} + r-\sqrt{\alpha^2 - (\sqrt{t})^2} <r,
$$
for $t$ small enough, in which case $z_s \in B(c, r) \subset S$.

Now we prove that the metric projection onto $\Gamma$ is injective on
the union of all $\alpha$-edges.  Indeed, assume that this is not the case, so there are two distinct points belonging to some (necessarily distinct) $\alpha$-edges, $x \in [X_{i_1} X_{i_2}]$ and $x' \in [X_{i_1'} X_{i_2'}]$, with the same
metric projection onto $\Gamma$, denoted $y = P_\Gamma(x) = P_\Gamma(x')$.
Let $\eta$ be the outward pointing unit normal vector at $y$.
For short, let $x_1 = X_{i_1}$, $x_2 = X_{i_2}$, $x_1' = X_{i_1'}$, $x_2' = X_{i_2'}$.  By the triangle inequality and
the fact that $\|x - x'\| \leq \dist(x, \Gamma) + \dist(x', \Gamma)$, and then \eqref{edge}, we have
\beq \label{max-dist-edges}
\max_{i,j \in\{1,2\}} \|x_i - x_j'\|
\leq \|x - x'\| + \|x_1 -x_2\| + \|x_1' -x_2'\|
\leq 2t + 2 \sqrt{t}\leq 3 \sqrt{t},
\eeq
when $t$ is small enough.
Also by the triangle inequality for angles and \eqref{edge},
\beq \label{injective-angle}
\angle((x_1 x_2), (x_1' x_2'))
\leq \angle((x_1 x_2), \vec{\Gamma}_y) + \angle(\vec{\Gamma}_y, (x_1' x_2'))
\leq 2 \sqrt{t}.
\eeq
Let $B$ and $B'$ denote the open balls of radius $\alpha$ circumscribing $[x_1 x_2]$ and $[x_1' x_2']$, respectively, and empty of data points.  Since all $\alpha$-edges are one-sided, these balls are uniquely defined.
Also, define $z'_e$ and $z'_s$ analogously to $z_e$ and $z_s$ above, but based on $x'_1$ and $x'_2$, instead of $x_1$ and $x_2$.
Using the same notation as above, we have $B=B(z_e,\alpha)$ and $B'=B(z_e',\alpha)$ and
\begin{equation}\label{balls-close}
\|z_e-z_e'\|\leq \|x_m-x_{m}'\|+\|\rho u -\rho' u'\|.
\end{equation}
Reasoning as in \eqref{max-dist-edges} above, we have $\|x_m-x_{m}'\| \le 3 \sqrt{t}$.  Also,
$$
\|\rho u -\rho' u'\|^2=\rho^2 +(\rho')^2-2\rho\rho'\langle u,u'\rangle.
$$
Using \eqref{edge}, $\rho^2 =\alpha ^2 -\|x_1-x_m\|^2\geq \alpha ^2-t$ and, similarly, $(\rho')^2 \geq \alpha^2-t$.  Moreover, by \eqref{injective-angle} and using again the inequality $\cos(a)\geq 1-a^2$ for any $a \in \bbR$, we get $\langle u, u'\rangle\geq 1-4t$. Hence,
$$
\|\rho u -\rho' u'\|^2 \leq 2\alpha^2 - 2 (\alpha^2-t) (1 - 4t) \le (8\alpha^2+2)t.
$$
Hence, the bound in \eqref{balls-close} leads to $\|z_e-z_e'\|\leq 3 \sqrt{t} + (8\alpha^2+2)^{1/2}\sqrt{t} = A_1 \sqrt{t}$ when $t$ is small enough, where $A_1$ is a constant.  Combining this bound with that
in \eqref{injective-angle}, and applying \lemref{two-edges}, we obtain that
$$
\max\{\angle((x x'), (x_1 x_2)), \angle((x x'), (x_1' x_2'))\} \leq A_2 \sqrt{t},
$$
where $A_2$ is a constant.
By the fact that $(x x')$ is parallel to $\eta$ \citep[Th.~4.18(12)]{MR0110078} and using \eqref{edge}, we also have
$$
\max\{\angle((x x'), (x_1 x_2)), \angle((x x'), (x_1' x_2'))\} \geq \frac{\pi}{2} -\sqrt{t}.
$$
We therefore have a contradiction when $t$ is small enough that all the derivations above apply and, in addition, $\sqrt{t} < \pi/(2A_2+2)$.
\end{proof}

\begin{rem} \label{rem:one-sided}
Any one-sided $\alpha$-edge shares each one of its endpoints with
another $\alpha$-edge.  Indeed, suppose $[x_1 x_2]$ is an $\alpha$-edge, so that there
exists $\zeta$ such that $x_1, x_2 \in \partial B(\zeta, \alpha)$ and $\cX_n \cap B(\zeta, \alpha) = \emptyset$.  In that case, let $B(\zeta, \alpha)$ pivot on $x_2$, as we did in the proof
of \prpref{exists} away from $x_1$.  Let $x_3$ denote the first data point that the ball hits.  Then $[x_2 x_3]$ is an $\alpha$-edge by construction.  If $x_2$ is not shared with any other $\alpha$-edge, then the ball pivots on $x_2$ away from $x_1$ until it touches $x_1$ from the other side.  That (open) ball is empty of data points inside, and together with the ball we started with, makes $[x_1 x_2]$ two-sided.
\end{rem}

\begin{prp} \label{prp:polygon}
Assume {\em \ref{setting}}.
For  any $\alpha \in (0,r)$, there is a constant $A > 0$ depending only on $(\alpha,r,\diam(S))$
such that, with probability at least $1 - A e^{-n/A}$, the union of all $\alpha$-edges is in one-to-one correspondence with $\partial S$ via the metric projection
onto $\partial S$.
\end{prp}

\begin{proof}
Let $\Gamma$ be a shorthand for $\partial S$ and $d = \diam(S)$, and let $C_\alpha$ denote the union of all $\alpha$-edges. Since $\Gamma$ is a (compact) one dimensional
manifold \citep{Walther99}, it is well-known that
each connected component of $\Gamma$ is a closed curve homeomorphic to the unit circle, see \cite[Thm. 5.27]{lee11}. We prove that this is
also the case for each connected component of $C_\alpha$.
We assume that the metric projection onto $\Gamma$, meaning $P_\Gamma$, is
injective on $C_\alpha$, that all $\alpha$-edges are
one-sided, that $C_\alpha\subset B(\Gamma,\alpha)$---so that $P_\Gamma$ is well-defined on $C_\alpha$---and that $C_\alpha\cap B(\Gamma_k,\alpha)\neq \emptyset$ for any connected component $\Gamma_k$ of $\Gamma$.
This event happens with probability at least $1 - A e^{-n/A}$ for some constant $A > 0$, by Propositions~\ref{prp:edge},~\ref{prp:exists} and~\ref{prp:injective}.
We prove that, under these circumstances, $C_\alpha$ is in one-to-one correspondence with $\Gamma$ via $P_\Gamma$.
Indeed, let $\Gamma_k$ be a connected component of $\Gamma$.
Let $[x_1x_2]$ be an $\alpha$-edge such that $[x_1x_2] \cap B(\Gamma_k, \alpha) \ne \emptyset$.
By assumption, there is a data point $x_3$ such that $[x_2x_3]$ is also an $\alpha$-edge.  Having constructed $[x_{a-1}x_a]$, let $x_{a+1}$ be a data point
such that $[x_a x_{a+1}]$ is an $\alpha$-edge.
Since $C_\alpha \subset B(\Gamma, \alpha) = \sqcup_\ell B(\Gamma_\ell, \alpha)$---where the union is of disjoint sets by
\cite[Rem.~4.15, (1)]{MR0110078}--- and the polygon $\cup_a [x_a x_{a+1}]$ is connected, necessarily, $\cup_a [x_a x_{a+1}] \subset B(\Gamma_k, \alpha)$.
Also, since the sequence $(x_a: a \ge 1)$ is made of finitely many data points, and $x_a \ne x_{a+1}$ for all $a$, there is $a,b \ge 1$ such that $x_a = x_{a+b+1}$, and we further may assume that $x_a, \dots, x_{a+b}$ are all distinct.  Therefore, by construction, $C = [x_a x_{a+1}]\cup \cdots \cup [x_{a+b-1} x_{a+b}]$ is a simple polygon made of $\alpha$-edges such that $C \subset B(\Gamma_k, \alpha)$.
In particular, the latter implies that $P_\Gamma(C) \subset \Gamma_k$, and since $C$ is homeomorphic to the unit circle and $P_\Gamma$ is continuous and injective on $C$, $P_\Gamma(C)$ is also homeomorphic to the unit circle.
This forces $P_\Gamma(C) = \Gamma_k$, due to $\Gamma_k$ being homeomorphic to the unit circle too.
Since all this is true for any $k$, meaning any connected component of $\Gamma$, we conclude therefore that $P_\Gamma : C_\alpha \to \Gamma$ is not only injective, but also surjective.
\end{proof}

\section{Proof of \thmref{main}} \label{sec:main}

We are now in a position to prove the main result, meaning, \thmref{main}. 
Let $\Gamma$ be a shorthand for $\partial S$ and let $C_\alpha$ denote the  union of all $\alpha$-edges.

By \prpref{edge} together with the union bound, and then \prpref{polygon}, for any $0<t \leq \min\{\alpha, 2\alpha^2/r\}$, with probability at least $1 -A_1 n^2 e^{-n t^{3/2}/A_1}$, for some constant $A_1 > 0$ depending only on $(\alpha, r, \diam(S))$, $C_\alpha$ is
in one-to-one correspondence with $\Gamma$ via the metric projection onto $\Gamma$, and
 satisfies $C_\alpha \subset B(\Gamma, t)$ and $\angle(C_\alpha, \Gamma) \leq \sqrt{t}$.  Note that, because $C_\alpha$ and $\Gamma$ are in one-to-one 
 correspondence, $C_\alpha \subset B(\Gamma, t)$ implies that $\Gamma \subset B(C_\alpha, t)$, so that $\cH(C_\alpha, \Gamma) \le t$.
We now apply \lemref{main-approx}, combined with the simple bounds $\cos a \geq 1 - a^2/2$, for $a>0$, and $(1-a)^{-1} \le 1 + 2 a$, valid when $0 < a \le 1/2$.  Assuming $t \le 1$, this yields
\[
\frac{\lambda(C_\alpha)}{\lambda(\Gamma)}\leq \frac{1+\frac{1}{r}\cH(C_\alpha,\Gamma)}{\cos(\angle (C_\alpha, \Gamma) ) }\leq \frac{1+t/r}{1-t/2}\leq (1+t/r)(1+t)\leq 1+(1+2/r)t
\]
and
\[
\frac{\lambda(C_\alpha)}{\lambda(\Gamma)}\geq 1-\frac{1}{r}\cH(C_\alpha,\Gamma)\geq 1-t/r.
\]
We get
\[
\left|\frac{\lambda(C_\alpha)}{\lambda(\Gamma)}-1\right|\leq(1+2/r)t.
\]
Hence, if $t\leq t_0 :=\min\{\alpha/2,2\alpha^2/r, 1\}$, we have
$$
\P\left(\left|\frac{\lambda(C_\alpha)}{\lambda(\Gamma)}-1\right|>(1+2/r)t\right)\leq A_1 n^2\exp(-n t^{3/2}/A_1).
$$
Then a change of variable concludes the proof of \thmref{main}.
\section{Numerical experiments}\label{sec:numerical}

In order to numerically check the conclusions of \thmref{main} we  performed a small simulation study. For the set $S$ we chose
the corona $\{x\in \mathbb{R}^{2}: 0.25\leq \|x\|\leq 1\}$.
In this case the value of $r$ is equal to $0.25$ (the radius of the hole) and $\lambda(\partial S)=2\pi(0.25+1)$.  The
selected sample sizes were $n=1000, 5000, 10000, 30000, 40000, 50000$. For each sample
size $n$, we simulated $M=1000$ samples from the uniform distribution
on $S$ and calculated the $\alpha$-shape for each sample.
The values of $\alpha$ were $0.05,0.1,0.15,0.2,0.24$, and the limit case $\alpha=r=0.25$.
Given $n$, $\alpha$, and sample $m \in \{1,\dots,M\}$, we computed the
sample $\alpha$-shape, denoted $C_\alpha^{n,m}$, using the R-package {\tt alphahull} of \cite{patrod10}, and then
 its perimeter $\lambda(C_\alpha^{n,m})$.
We estimated the expected error and bias by
$$e_\alpha(n)= \frac1M \sum_{m=1}^{M}|\lambda(C_\alpha^{n,m})-\lambda(\partial S)| \quad \text{and} \quad b_\alpha(n)=\frac1M \sum_{m=1}^{M} \lambda(C_\alpha^{n,m})-\lambda(\partial S),$$ respectively.
Let $s_\alpha(n)$ denote the sample standard deviation of $\{\lambda(C_\alpha^{n,m}), m =1,\ldots, M\}$.


\bitem \setlength{\itemsep}{0in}
\item Among the $\alpha$'s that we tried, the estimator performs best at $\alpha = 0.2$. It does not seem that, asymptotically, the best $\alpha$ converges to $r$.  For instance, the ratio $e_{0.24}(n)/e_{0.2}(n)$ is around 6.7 for $n\geq 30000$.

\item
\figref{error} shows the error versus sample size in log-log scale for
$\alpha=0.1,0.2, 0.24,0.25$. It can be seen that the error corresponding to $\alpha=r$ does no go to zero whereas $\alpha=0.2$ always outperform the other
considered values of $\alpha$. The trend for large values of $n$ is clearly linear and the slope is close to $-2/3$ as \thmref{main} predicts.
This is particularly true when $\alpha=0.2$ (our best choice), where fitting a line by least squares yields a slope of $-0.67$, with (Student) 95\%-confidence interval of $(-0.73,-0.62)$, and an R-squared exceeding 0.99.

\item For the limit case $\alpha=r$, the bias, $b_{\alpha}(n)$ does not go to zero as the sample size increases. The error $e_r{(n)}$ is approximately
equal to 0.18; see \figref{error}. This shows, from the numerical point of view, that the perimeter of
the $\alpha$-shape is not a consistent estimator of the $\lambda(\partial S)$ for $\alpha=r$. The main problem here
is that the length of the $\alpha$-edges does not go to zero, as \prpref{edge} states for $\alpha<r$.

\item The convergence rate of the standard deviation seems to be higher that $-2/3$.  In fact, we have reasons to believe that the slope is of order $n^{-5/6}$.  This is confirmed numerically.  Indeed, if we fit a line to the log-log plot of $s_{0.2}(n)$, we get a slope with (Student) 95\%-confidence interval of $(-0.86,-0.82)$.  So, asymptotically, it seems that the error is dominated by the bias.  This suggests that reducing the bias of the estimator could lead to improve the convergence rate of the method.

\item The random variable $\lambda(C_{\alpha})$ seems to be asymptotically normal.  For the greatest considered $n = 50,000$, the sample $\{\lambda(C_\alpha^{n,m}), m =1,\ldots, M\}$ passes the Shapiro-Wilks normality test for several values of $\alpha$. For instance, for $\alpha=0.2$, we got a p-value of 0.82.
\eitem

\begin{figure}[htbp]
\centering
\scalebox{0.75}{\includegraphics[width=0.6\linewidth,height=0.5\linewidth]{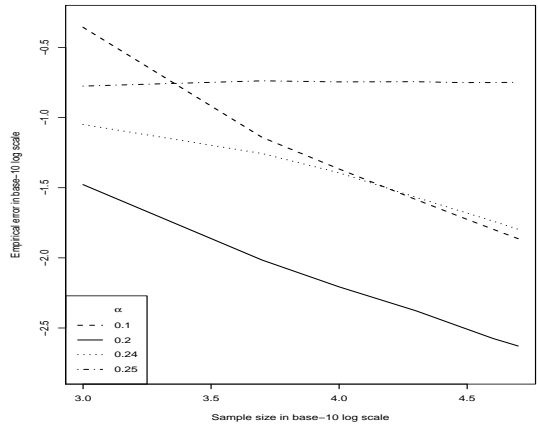}} \
\caption{Plot of error versus sample size, in log-log scale.  The error
corresponding to $\alpha = r = 0.25$ does not converge to zero.
For values of $\alpha < r$, the plots show asymptotic slopes which are all very close to $-2/3$, as \thmref{main} predicts.
}
\label{fig:error}
\end{figure}

\section{Discussion}\label{sec:discussion}

We discuss a number of extensions and open problems.

{\em Extensions.}
Our arguments extend more or less trivially to other sampling distributions.  It is completely straightforward to see that \thmref{main} applies verbatim to a sampling
distribution which has a density with respect to the uniform distribution which is bounded away from zero near the boundary of $S$.  A little less obvious is an
extension to the case where this density converges to zero at some given rate near the boundary, which ends up impacting the rate of convergence of
our estimator.  In any case, our estimator remains consistent.  The same results carry over to the case where $\partial S$ has a finite number of `kinks', i.e., points where the reach is infinite.

{\em Choice of tuning parameter.}
The estimator depends on knowledge of $r$, or at least a lower bound on $r$, since any $\alpha \in (0,r)$ fixed appears to yield the
convergence rate in $n^{-2/3}$.  Choosing $\alpha$ automatically, therefore, requires an estimate on the size of $r$.  This is done in recent work by \cite{r-estimation}.
Suppose we have an estimator $\hat r_n$ such that $r/2 \le \hat r_n \le 3r/2$ with high probability.  We speculate that the convergence bound obtained in \thmref{main} with $\alpha$ chosen equal to $\hat r_n/4$ remains valid, albeit with a different multiplicative constant.

{\em Finer asymptotics.}
\cite{MR1641826} were able to compute the exact asymptotic expected value and variance of the perimeter of the convex hull of a sample, and also to show an asymptotic normal limit.  An open problem would be to do the same here.  Our numerical experiments lead us to speculate that our estimator is also normal in the large-sample limit.

{\em Minimax rate.}
We conjecture that the rate that our estimator achieves, i.e., $n^{-2/3} {\rm polylog}(n)$, is not minimax optimal, not even in the exponent.  Indeed, we learn in \citep[Chap 8]{MR1226450} that for the problem of estimating the area (in the context of binary images), an estimator obtained from computing the area of an optimal set estimator (for the symmetric difference metric, and the $\alpha$-convex hull is such an estimator) only achieves the rate $n^{-2/3}$, while the optimal rate is $n^{-5/6}$ with the assumptions we make here.
It is very reasonable to infer that the same is true for the more delicate problem of perimeter estimation.  In fact, \cite{kim2000estimation}  show that $n^{-5/6}$ is (up to a poly-logarithmic factor) the minimax rate for perimeter estimation of a horizon (also in the context of binary images).  

{\em Higher dimensions.}
Our setting is that of a set $S$ in two dimensions.  How about higher dimensions?  The problem would be to estimate the $(d-1)$-volume of the boundary of a set $S \subset \bbR^d$, under the same conditions, and
the estimator would be the $(d-1)$-volume of the $\alpha$-shape of $\cX_n$, which is the union of all the $\alpha$-faces.  We say that $X_{i_1}, \dots, X_{i_d}$ form an $\alpha$-face if they are affine-independent and there is an open ball $B$ of radius $\alpha$ such that $X_{i_1}, \dots, X_{i_d} \in \partial B$ and $B \cap \cX_n = \emptyset$.
Most of the auxiliary lemmas and propositions can be extended to the general framework. However, we have no idea how to extend \prpref{polygon}.

{\em The $\alpha$-convex hull.}
Our results apply to the $\alpha$-convex hull of the sample.
This is because, with high probability, it shares the same vertices as the $\alpha$-shape (by \prpref{similar}).  When this is the case, the former is the union of arcs of radius $\alpha$ with base the $\alpha$-edges.  In particular, if an $\alpha$-edge is of length $\ell$, then the length of that arc is $2\alpha \sin^{-1}(\ell/(2\alpha)) = \ell + O(\ell^3)$.
By \prpref{edge} and an application of the union bound, the largest $\alpha$-edge is of order $O_P(\log(n)/n)^{2/3}$.  We conclude that the ratio between the
perimeters of the $\alpha$-convex hull and of the $\alpha$-shape is of order $1 + O_P(\log(n)/n)^{4/3}$.
We note, however, that the perimeter of the $r$-convex hull is consistent while the perimeter of the
$r$-shape is not necessarily so.  Our results require $\alpha < r$.

\subsection*{Acknowledgements}

EAC was partially supported by a grant from the US National Science Foundation (DMS-0915160).
ARC was partially supported by Project MTM2008--03010 and MTM2013-41383-P from the
Spanish Ministry of Science
and Innovation, and by the IAP network StUDyS (Developing crucial Statistical methods for Understanding major complex Dynamic Systems in natural, biomedical and social sciences) of the Belgian Science Policy.

\bibliographystyle{chicago}
\bibliography{perimeter}

\end{document}